\documentclass[12pt]{amsart}

\usepackage{amssymb,latexsym}

\usepackage{enumerate}

\makeatletter

\@namedef{subjclassname@2010}{

  \textup{2010} Mathematics Subject Classification}

\makeatother
\newtheorem{thm}{Theorem}[section]

\newtheorem{cor}[thm]{Corollary}
\newtheorem{lem}[thm]{Lemma}
\newtheorem{pro}[thm]{Proposition}
\theoremstyle{definition}

\newtheorem{rem}[thm]{Remark}

\numberwithin{equation}{section}

\newcommand{\pr}{\mathbb{P}}
\newcommand{\ex}{\mathbb{E}}
\newcommand{\re}{\textup{Re}}
\newcommand{\im}{\textup{Im}}
\newcommand{\F}{\mathcal{F}(x)}
\newcommand{\Fstar}{\mathcal{F}^*(x)}
\newcommand{\E}{\mathcal{E}(x)}
\newcommand{\g}{\gamma_{\mathbb{Q}(\sqrt D)}}
\newcommand{\gr}{\gamma_{\textup{rand}}(X)}
\newcommand{\lap}{\mathcal{L}_x}
\newcommand{\lapr}{\mathcal{L}_{\textup{rand}}}
\begin{document}

\baselineskip=16pt

\title{The distribution of Euler-Kronecker constants of quadratic fields}

\author[Youness Lamzouri]{Youness Lamzouri}

\address{Department of Mathematics and Statistics,
York University,
4700 Keele Street,
Toronto, ON,
M3J1P3
Canada}

\email{lamzouri@mathstat.yorku.ca}

\date{}

\begin{abstract}  We investigate the distribution of large positive (and negative) values of the Euler-Kronecker constant $\g$ of the quadratic field $\mathbb{Q}(\sqrt{D})$ as $D$ varies over fundamental discriminants $|D|\leq x$. We show that the distribution function of these values  is very well approximated by that of an adequate probabilistic random model in a large uniform range. The main tools are an asymptotic formula for the Laplace transform of $\g$ together with a careful saddle point analysis. 
\end{abstract}


\thanks{The author is partially supported by a Discovery Grant from the Natural Sciences and Engineering Research Council of Canada.}

\maketitle

\section{Introduction}

Let $K$ be an algebraic number field, $\mathcal{O}_K$ be its ring of integers and $N(\mathfrak{a})$ denote the norm of an ideal $\mathfrak{a}$ in $\mathcal{O}_K$. The Dedekind zeta function of $K$ is defined for $\re(s)>1$ by
$$\zeta_K(s)=\sum_{\mathfrak{a}}\frac{1}{N(\mathfrak{a})^s}= \prod_{\mathfrak{p}}\left(1-\frac{1}{N(\mathfrak{p})^s}\right)^{-1}, $$
where $\mathfrak{a}$ ranges over non-zero ideals and $\mathfrak{p}$ ranges over the prime ideals in $\mathcal{O}_K$. It is known that $\zeta_K(s)$ has an analytic continuation to $\mathbb{C}\setminus\{1\}$ and a simple pole at $s=1$ with residue $\alpha_K$. The well-known class number formula relates $\alpha_K$ to several algebraic invariants of $K$, including the discriminant, class number and regulator of $K$. 

The \emph{Euler-Kronecker} constant (or invariant) of $K$ is defined by
$$\gamma_K=\lim_{s\to 1} \left(\frac{\zeta'_K(s)}{\zeta_K(s)}+\frac{1}{s-1}\right).$$
Moreover, if the Laurent series expansion of $\zeta_K(s)$ is 
$$ \zeta_K(s)= \frac{\alpha_K}{s-1} + c_0(K)+ c_1(K) (s-1)+ c_2(K)(s-1)^2\cdots,$$
then 
$\gamma_K=c_0(K)/\alpha_K.$
Note that when $K=\mathbb{Q}$, we have $\gamma_{K}=\gamma$, where $\gamma=0.577...$ is the Euler-Mascheroni constant.

The Euler-Kronecker constant was first introduced and studied by Ihara in \cite{Ih1} and \cite{Ih2}. In particular, Ihara proved in \cite{Ih1} that if $d_K$ is the discriminant of $K$ then 
$$ -\frac12\log {|d_k|}\leq \gamma_K\leq 2\log\log |d_K|, $$ 
where the upper bound is conditional on the Generalized Riemann hypothesis GRH. Tsafsman \cite{Ts} showed that the lower bound is optimal up to a constant, and hence that the maximal order of $|\gamma_K|$ is $\asymp \log |d_K|.$ However, Ihara \cite{Ih1} proved that this order is much smaller if the degree of $K$ is small. 

When $K$ is the cyclotomic field $K(q):=\mathbb{Q}\big(e^{2\pi i/q}\big)$, Ihara \cite{Ih1} showed that $\gamma_{K(q)}=O(\log^2 q)$ assuming GRH, and this bound was improved to $O(\log q\log\log q)$ by Badzyan \cite{Ba}.  Murty \cite{Mu} proved an upper bound for the first moment of $\gamma_{K(q)}$, which was refined to an asymptotic formula by  Fouvry \cite{Fo}, who showed that the average order of $\gamma_{K(q)}$ is $\log Q$. In the case where $q$ is prime, Ford, Luca and Moree \cite{FLM} studied $\gamma_{K(q)}$ and showed that it appears in the asymptotic expansion of the number of integers $n\leq x$ for which $\varphi(n)$ is not divisible by $q$, where $\varphi$ is the Euler $\varphi$-function. 

In the special case where $K=\mathbb{Q}(\sqrt{D})$ is a quadratic field, we know that the corresponding Dedekind zeta function factorizes as $\zeta_K(s)=\zeta(s) L(s,\chi_D)$, where $\chi_D(n)=(D/n)$ is the Kronecker symbol. Therefore
$$\g= \gamma+ \frac{L'(1,\chi_D)}{L(1,\chi_D)}.$$
When $\mathbb{Q}(\sqrt{D})$ is imaginary, the Kronecker limit formula expresses $\g$ in terms of special values
of the Dedekind $\eta$-function (see Section 2.2 of \cite{Ih1}). 

In \cite{Ih1}, Ihara proved that under GRH we have 
$$
 |\g|\leq (2+o(1))\log\log |D|.
$$ Using a zero density result of Heath-Brown \cite{HB}, we show in Corollary \ref{ASBound} below that this bound is attained for almost all fundamental discriminants. More precisely, we prove that  for all but at most $O(x^{\epsilon})$ fundamental discriminants $D$ with $|D|\leq x$ we have 
$$\g \ll_{\epsilon}\log\log |D|.$$
On the other hand, Mourtada and Murty \cite{MoMu} proved that there are infinitely many $D$ for which 
$$
\pm  \g \geq \log\log |D| + O(1).$$ They also showed that this bound can be improved  to $\log\log |D|+\log\log\log |D|+O(1)$ under GRH.

In analogy to $L(1,\chi_D)$, we expect that for all fundamental discriminants $D$ with $|D|\leq x$ we have 
\begin{equation}\label{TrueRange}
|\g| \leq \log\log x +\log\log\log x+O(1),
\end{equation} 
so that the true order of extreme values of $\g$ is closer to the omega results of Mourtada-Murty rather than the conditional $O$-result of Ihara. Our Theorem \ref{AsympDistrib} below gives strong support for this conjecture (see Remark \ref{Support} below).

To investigate the distribution of the Euler-Kronecker constant $\g$, our strategy consists in constructing an adequate probabilistic random model for these values.  
Let $\{X(p)\}_{p \text{ prime}}$ be a sequence of independent random variables, indexed by the primes, and taking the values $1, -1$ and $0$ with the following probabilities
$$ \pr(X(p)=a)= \begin{cases} \frac{p}{2(p+1)} & \text{ if } a=\pm 1,\\
\frac{1}{p+1}  & \text{ if } a=0.\\
\end{cases}
$$
We extend the $X(p)$ multiplicatively to all positive integers by setting $X(1)=1$ and 
$ X(n):= X(p_1)^{a_1}\cdots X(p_k)^{a_k}, $ if $n= p_1^{a_1}\cdots p_k^{a_k}.$
These random variables were first introduced by Granville and Soundararajan \cite{GrSo} to study the distribution of $L(1,\chi_D)$. The reason for this choice over the simpler $\pm 1$ with probability $1/2$ is that for odd primes $p$, fundamental discriminants $D$ lie in one of $p^2-1$ residue classes mod $p^2$ so that $\chi_D(p)=0$ for $p-1$ of these classes, and the remaining $p(p-1)$ residue classes split equally into $\pm 1$ values (for $p=2$ one can check that the values $0,\pm 1$ occur equally often). 
We shall compare the distribution of $\g$, as $D$ varies among fundamental discriminants $|D|\leq x$, to that of the following probabilistic random model:
$$\gr:= \gamma-\sum_{n=1}^{\infty} \frac{\Lambda(n) X(n)}{n}=\gamma-\sum_{p}\frac{(\log p)X(p)}{p-X(p)}.$$
Since $\ex(X(n))=0$ unless $n$ is a square (see \eqref{ortho} below), and $\sum_{n\geq 2}(\log n)^2/n^2<\infty$, then it follows from Kolmogorov three series theorem that $\gr$ is almost surely convergent.

Here and throughout, we denote by $\F$ the set of all fundamental discriminants $D$ with $|D|\leq x$. Note that $|\F|= 6 x/\pi^2 +O(\sqrt{x}).$   Our main result shows that the distribution of $\g$ is very well approximated by that of the random variable $\gr$ uniformly in nearly the whole conjectured range \eqref{TrueRange}. 

\begin{thm}\label{MainTheorem}
Let $x$ be large. There exists a positive constant $C$ such that uniformly in the range $1\leq \tau\leq \log\log x-2\log\log\log x-C$, we have 
$$\frac{1}{|\F|}\big|\{D\in \F: \g >\tau \}\big|= \pr\big(\gr>\tau\big)\left(1+O\left(\frac{e^{\tau}(\log\log x)^3}{\tau\log x}\right)\right),$$
and 
$$ \frac{1}{|\F|}\big|\{D\in \F: \g <-\tau \}\big|= \pr\big(\gr<-\tau\big)\left(1+O\left(\frac{e^{\tau}(\log\log x)^3}{\tau\log x}\right)\right).$$
\end{thm} 
Since $L'/L(1,\chi_D)=\g -\gamma$, Theorem 1.1 can be rephrased in terms of the logarithmic derivative of quadratic Dirichlet $L$-functions at $s=1$. The values of logarithmic derivatives of $L$-functions have been studied by Ihara and Matsumoto \cite{IhMa}, and Ihara, Murty and Shimura \cite{IMS} in the case of Dirichlet $L$-functions, and by Cho and Kim \cite{ChKi} in the case of Artin $L$-functions. In particular, Ihara and Matsumoto \cite{IhMa} showed that as $\chi$ varies over non principal characters modulo a prime $q$,  $L'/L(1,\chi)$ has a limiting distribution as $q\to\infty$. However, Theorem \ref{MainTheorem} is the first result that gives precise information on the distribution of logarithmic derivatives of $L$-functions at $s=1$ with such a great uniformity. We should also note that with a slight modification of our method we can obtain similar results for the distribution of $|\zeta'/\zeta(1+it)|$, and that of $|L'/L(1, \chi)|$ as $\chi$ varies over non-principal characters modulo a large prime $q$. To construct the probabilistic random model in these cases we take the $\{X(p)\}_p$ to be uniformly distributed on the unit circle.

Our next task is to study the asymptotic behavior of the distribution functions $\pr\big(\gr>\tau\big)$ and $\pr\big(\gr<-\tau\big)$  in terms of $\tau$, when $\tau$ is large. We achieve this by a careful saddle point analysis. In particular, we show that these distribution functions are double exponentially decreasing in $\tau$. 

\begin{thm}\label{ExponentialDecay}
For large $\tau$ we have  
$$ \pr(\gr>\tau)=\exp\left(-\frac{e^{\tau-A_1}}{\tau}\left(1+ O\left(\frac{\log\tau}{\tau}\right)\right)\right),$$
and 
$$ \pr(\gr<-\tau)=\exp\left(-\frac{e^{\tau-A_2}}{\tau}\left(1+O\left(\frac{\log\tau}{\tau}\right)\right)\right),$$
where 
$$A_1:=A_0+2\frac{\zeta'(2)}{\zeta(2)}, \text{ and }
A_2:=A_0- 2\gamma,
$$
and $$ A_0:= \int_0^1\frac{\tanh(t)}{t}dt + \int_1^{\infty}\frac{\tanh(t)-1}{t}dt.$$
\end{thm}
Combining Theorems \ref{MainTheorem} and \ref{ExponentialDecay} we deduce that the same asymptotic estimate holds for the distribution function of $\g$ uniformly for $\tau$ in the range $1\ll \tau \leq \log\log x-2\log\log\log x-C.$

\begin{thm}\label{AsympDistrib}
Let $x$ be large. There exists a positive constant $C$ such that uniformly in the range $1\ll \tau\leq \log\log x-2\log\log\log x-C$, we have 
\begin{equation}\label{AsymptoticEstimate1}
\frac{1}{|\F|}\big|\{D\in \F: \g >\tau \}\big|= \exp\left(-\frac{e^{\tau-A_1}}{\tau}\left(1+ O\left(\frac{\log\tau}{\tau}\right)\right)\right),
\end{equation}
and 
\begin{equation}\label{AsymptoticEstimate2}
\frac{1}{|\F|}\big|\{D\in \F: \g <-\tau \}\big|= \exp\left(-\frac{e^{\tau-A_2}}{\tau}\left(1+ O\left(\frac{\log\tau}{\tau}\right)\right)\right).
\end{equation}
\end{thm}

\begin{rem}\label{Support} Note that the asymptotic estimate on the right hand side of \eqref{AsymptoticEstimate1} (or \eqref{AsymptoticEstimate2}) becomes $<1/|\F|$ if $\tau>\log\log x+\log\log\log x+C_0$ for some constant $C_0$. Therefore, if the asymptotic estimates in \eqref{AsymptoticEstimate1} and \eqref{AsymptoticEstimate2} were to persist in this full viable range, then one would deduce that $|\g|\leq \log\log |D|+\log\log\log |D|+O(1).$
\end{rem}

 In  \cite{GrSo}, Granville and Soundararajan investigated the distribution of $L(1,\chi_D)$ and proved that uniformly for $\tau$ in the range $1\ll \tau\leq \log\log x +O(1)$ we have 
$$\frac{1}{|\F|}\big|\{D\in \F: L(1,\chi_D) >e^{\gamma}\tau \}\big|= \exp\left(-\frac{e^{\tau-A_0}}{\tau}\left(1+O\left(\frac{1}{\tau}\right)\right)\right).$$
Their method relies upon careful analysis of  large complex moments of $L(1,\chi_D)$. In her thesis, Mourtada \cite{Mo} remarked that it is a difficult problem to compute complex moments of $\g$.  Instead, our approach relies on computing the Laplace transform of $\g$ (defined as the average of $\exp(s\cdot \g)$ over $D\in \F$) using only asymptotics for integral moments of $\g$. 
We should also note that in comparison to the treatment for $L(1, \chi_D)$, there is an additional technical difficulty in our case which comes from the fact that $\exp({\g})$ grows much faster than $L(1,\chi_D)$. To overcome this difficulty, we compute the Laplace transform of $\g$ after first removing the contribution of a small set of ``bad'' discriminants $D$, namely those for which $\g$ might be large.

\begin{thm}\label{AsympLaplace} Given $0<\epsilon\leq 1/2$ there exists a constant $C_{\epsilon}>0$ and 
a set of fundamental discriminants $\E\subset \F$ with $|\E|=O\left(x^{\epsilon}\right)$, such that for all complex numbers $s$ with $|s|\leq C_{\epsilon}\log x/(\log\log x)^2$ we have 
$$\frac{1}{|\F|}\sum_{D\in \F\setminus \mathcal{E}(x)}\exp\left(s \cdot \g\right) = \ex\Big(\exp\big(s\cdot\gr\big) \Big)+ 
O\left(\exp\left(-\frac{\log x}{50\log\log x}\right)\right).$$
\end{thm}

To prove this result we show that large integral moments of $\g$ are very close to those of the random model $\gr$.  For a fixed natural number $k$, asymptotic formulae for the $k$-th moment of $\g$ have been obtained by Mourtada and Murty in \cite{MoMu}, building on an earlier work of Ihara, Murty and Shimura \cite{IMS}. However, the significant feature of our result is the uniformity in the range of moments. 

\begin{thm}\label{MomentsGamma} 
For all positive integers $k$ with $k\leq \log x/(50\log\log x)$ we have 
$$\frac{1}{|\F|}\sum_{D\in  \mathcal{F}^*(x)} \big(\g\big)^k= \ex\left(\big(\gr\big)^k\right) +O\left(x^{-1/30}\right),$$
where 
$ \mathcal{F}^*(x)$ denotes the set of fundamental discriminants $D\in \F$ such that $L(s,\chi_D)$ has no Siegel zeros.
\end{thm}
\begin{rem} Note that if $L(s, \chi_D)$ has a Siegel zero, we could have $
\g$ as large as $q^{\epsilon}$, so that when $k$ is large, the $k$-th moment of $\g$ would be heavily affected by the contribution of this particular character. This justifies  the condition $D\in \F^*$ in Theorem \ref{MomentsGamma}. Furthermore, it is known that these characters if they exist must be very rare, in particular we have $|\F|-|\F^*|\ll \log x$ (see for example  \cite{Da}). 

\end{rem}
The paper is organized as follows: In Section 2 we investigate the moments of $\g$ and prove Theorem \ref{MomentsGamma}. This result is then used to study the Laplace transform of $\g$ and prove Theorem \ref{AsympLaplace} in Section 3. In Section 4 we study the Laplace transform of the random model $\gr$ and prove an asymptotic estimate for it.  We then relate the distribution function of $\gr$ to its Laplace transform and prove Theorem \ref{ExponentialDecay} in Section 5. Finally, in Section 6 we combine all these results to derive Theorem \ref{MainTheorem}.

\section{Large moments of $\g$: proof of Theorem \ref{MomentsGamma}}

\noindent For any positive integer $k$, we define $$\Lambda_{k}(n)=\sum_{\substack{n_1,n_2,\dots,n_k\geq 1\\ n_1n_2\cdots n_k=n}}\Lambda(n_1)\Lambda(n_2)\cdots\Lambda(n_k).$$
Then for all complex numbers $s$ with $\re(s)>1$
we have
$$ \left(-\frac{L'}{L}(s,\chi_D)\right)^k=\sum_{n=1}^{\infty} \frac{\Lambda_k(n)}{n^s}\chi_D(n).$$
Moreover, note that
\begin{equation}\label{boundLambda}
\Lambda_{k}(n) \leq \left(\sum_{m|n}\Lambda(m)\right)^k= (\log n)^k.
\end{equation}
We shall extract Theorem \ref{MomentsGamma} from the following result, which gives an asymptotic formula for large integral moments of $-L'/L(1,\chi_D)$. 
\begin{thm}\label{MomentsEuler} 
For all positive integers $k$ with $k\leq \log x/(50\log\log x)$ we have 
$$\frac{1}{|\F|}\sum_{D\in \mathcal{F}^*(x)} \left(-\frac{L'}{L}(1,\chi_D)\right)^k= \sum_{m=1}^{\infty}
\frac{\Lambda_{k}(m^2)}{m^2}\prod_{p|m}\left(\frac{p}{p+1}\right) +O\left(x^{-1/20}\right).$$
\end{thm}
First, we need the following lemma, which provides a bound for $L'/L(s,\chi_D)$ when $s$ is far from a zero of $L(z, \chi_D)$. 

\begin{lem}\label{BoundLD}
Let $t$ be a real number and suppose that $L(z,\chi_D)$ has no zero for $\re(z)>\sigma_0$ and $|\im(z)|\leq |t|+1$, then for any $\sigma>\sigma_0$ we have 
$$\frac{L'}{L}(\sigma+it,\chi_D)\ll \frac{\log(D(|t|+2))}{\sigma-\sigma_0}.$$
\end{lem}
\begin{proof}
Let $\rho$ runs over the non-trivial zeros of $L(s,\chi)$. Then it follows from equation (4) of Chapter 16 of Davenport \cite{Da} that 
\begin{align*}
\frac{L'}{L}(\sigma+it,\chi_D)
&= \sum_{\substack{\rho\\ |t-\im(\rho)|<1}}\frac{1}{\sigma+it-\rho}+ O\big(\log(D(|t|+2))\big)\\
&\ll \frac{1}{\sigma-\sigma_0}\left(\sum_{\substack{\rho\\ |t-\im(\rho)|<1}} 1\right) + \log(D(|t|+2))\\
&\ll \frac{\log(D(|t|+2))}{\sigma-\sigma_0},
\end{align*}
as desired.

\end{proof}

The key ingredient in the proof of Theorem \ref{MomentsEuler} is the following result which shows that we can approximate large powers of $-L'/L(1,\chi_D)$ by short Dirichet polynomials, if $L(s,\chi_D)$ has no zeros in a certain region to the left of the line $\re(s)=1$. 
\begin{pro}\label{ApproximationLarge} Let $0<\delta<1/2$ be fixed, and $D$ be a fundamental discriminant with $|D|$  large. Let $y\geq (\log |D|)^{10/\delta}$ be a  real number and $k\leq 2\log |D|/\log y$ be a positive integer. If $L(s,\chi_D)$ is non-zero for $\re(s)>1-\delta$ and $|\im(s)|\leq y^{k\delta}$, then we have 
$$ \left(-\frac{L'}{L}(1,\chi_D)\right)^k=\sum_{n\leq y^k} \frac{\Lambda_k(n)}{n}\chi_D(n)+O_{\delta}\Big(y^{-k\delta/4}\Big).$$
\end{pro}

\begin{proof}
Without loss of generality, suppose that $y^k\in \mathbb{Z}+1/2$. Let $c=1/(k\log y)$, and $T$ be a large real number to be chosen later. Then by Perron's formula, we have 
$$ \frac{1}{2\pi i}\int_{c-iT}^{c+iT} \left(-\frac{L'}{L}(1+s,\chi_D)\right)^k \frac{y^{ks}}{s}ds= 
\sum_{n\leq y^k} \frac{\Lambda_k(n)}{n}\chi_D(n)+ O\left(\frac{y^{kc}}{T}\sum_{n=1}^{\infty} \frac{\Lambda_k(n)}{n^{1+c}|\log(y^k/n)|}\right).$$
To bound the error term of this last estimate, we split the sum into three parts: $n\leq y^k/2$, $y^k/2<n<2y^k$ and $n\geq 2y^k$. The terms in the first and third parts satisfy $|\log(y^k/n)|\geq \log 2$, and hence their contribution is 
$$\ll \frac{1}{T} \sum_{n=1}^{\infty}\frac{\Lambda_k(n)}{n^{1+c}}=\frac{1}{T} \left(\sum_{n=1}^{\infty}\frac{\Lambda(n)}{n^{1+c}}\right)^k\ll \frac{(2k\log y)^k}{T},$$
by the prime number theorem. To handle the contribution of the terms $y^k/2<n<2y^k$, we put $r=n-y^k$, and use that $|\log(y^k/n)|\gg |r|/y^k$. In this case, we have $\Lambda_k(n)\leq (\log n)^k\leq (2k\log y)^k$, and hence the contribution of these terms is 
$$\ll\frac{(2k\log y)^{k}}{Ty^{k}}\sum_{|r|\leq y^{k}}\frac{y^k}{|r|}\ll \frac{(2k\log y)^{k+1}}{T}.$$
We now choose $T=y^{k\delta/2}$ and move the contour to the line $\re(s)=-\delta/2$. By our assumption, we only encounter a simple pole at $s=0$ which leaves a residue $(-L'/L(1,\chi_D))^k$. Therefore, we deduce that 
$$ \frac{1}{2\pi i}\int_{c-iT}^{c+iT} \left(-\frac{L'}{L}(1+s,\chi_D)\right)^k \frac{y^{ks}}{s}ds= \left(-\frac{L'}{L}(s,\chi_D)\right)^k+ E_1,$$
where 
\begin{align*}
E_1&=\frac{1}{2\pi i} \left(\int_{c-iT}^{-\delta/2-iT}+ \int_{-\delta/2-iT}^{-\delta/2+iT}+ \int_{-\delta/2+iT}^{c+iT}\right) \left(-\frac{L'}{L}(1+s,\chi_D)\right)^k \frac{y^{ks}}{s}ds\\
& \ll_{\delta} \frac{(\log (|D|T))^k}{T}+ y^{-k\delta/2}\left(\frac{\log (|D|T)}{\delta}\right)^{k+1}\\
&\ll_{\delta} y^{-k\delta/4},
\end{align*}
by Lemma \ref{BoundLD}. Finally, since $(2k\log y)^{k+1}/T\ll y^{-k\delta/4}$, the result follows.  
\end{proof}

Now, using a zero density estimate due to Heath-Brown (see equation \eqref{ZeroDensity} below), we deduce from Proposition \ref{ApproximationLarge} that large powers of $-L'/L(1,\chi_D)$ can be approximated by short Dirichlet polynomials for almost all fundamental discriminants $D$ with $|D|\leq x$. 

\begin{cor}\label{AAApproximation}
Let $k$ be a positive integer such that $k\leq \log x/(50(\log\log x))$.  For all except $O(x^{3/4})$ fundamental discriminants $D$ with $|D|\leq x$ we have 
$$ \left(-\frac{L'}{L}(1,\chi_D)\right)^k=\sum_{n\leq x} \frac{\Lambda_k(n)}{n}\chi_D(n)+O\Big(x^{-1/20}\Big).$$
\end{cor}

\begin{proof} Let $N(\sigma, T, \chi_D)$ denote the number of zeros of $L(s, \chi_D)$ in the rectangle $ \sigma<\re(s)\leq 1$ and $|\im(s)|\leq T$. Health-Brown \cite{HB} showed that
\begin{equation}\label{ZeroDensity}
 \sum_{D\in \F} N(1-\delta, T, \chi_D)\ll_{\epsilon} (xT)^{\epsilon} x^{3\delta/(1+\delta)} T^{(1+2\delta)/(1+\delta)}.
\end{equation}
Choosing $\delta=1/5$, we deduce that for all except $O(x^{3/4})$ fundamental discriminants $D$ with $|D|\leq x$, $L(s, \chi_D)$ does not vanish in the region $\re(s)>1-\delta$ and $|\im(s)|\leq x^{\delta}$. We now take $y=x^{1/k}$  in Proposition \ref{ApproximationLarge}, to obtain that for all except $O(x^{3/4})$ fundamental discriminants $D$ with $\sqrt{x}\leq |D|\leq x$ we have 
$$ \left(-\frac{L'}{L}(1,\chi_D)\right)^k=\sum_{n\leq x} \frac{\Lambda_k(n)}{n}\chi_D(n)+O\Big(x^{-1/20}\Big),$$
as desired.

\end{proof}
We also deduce from Proposition \ref{ApproximationLarge} that $\g\ll \log\log |D|$ for almost all fundamental discriminants $|D|\leq x$.

\begin{cor}\label{ASBound}
Let $\epsilon>0$. Then for all but $O(x^{\epsilon})$ fundamental discriminants $|D|\leq x$ we have
$$ \g\ll_{\epsilon} \log\log D.$$
\end{cor}
\begin{proof}
Taking $\delta=\epsilon/5$, $k=1$ and $y=(\log |D|)^{50/\epsilon}$ in Proposition \ref{ApproximationLarge} and using \eqref{ZeroDensity} as in the proof of Corollary \ref{AAApproximation} we deduce that for all except $O(x^{\epsilon})$ fundamental discriminants $D$ with $|D|\leq x$, we have 
\begin{align*}
\g&=\gamma+ \frac{L'}{L}(1,\chi_D)=\gamma-\sum_{n\leq y} \frac{\Lambda(n)}{n}\chi_D(n)+O\Big(y^{-\epsilon/20}\Big)\\
&\ll_{\epsilon} \log\log |D|.
\end{align*}
\end{proof}

\begin{proof}[Proof of Theorem \ref{MomentsEuler}]
Let $\E$ be the exceptional set in Corollary \ref{AAApproximation}. Then it follows from this result that
$$ 
\frac{1}{|\F|}\sum_{D\in \Fstar\setminus \E} \left(-\frac{L'}{L}(1,\chi_D)\right)^k
= \frac{1}{|\F|}\sum_{D\in \Fstar\setminus \E} \sum_{n\leq x} \frac{\Lambda_k(n)}{n}\chi_D(n)+O\Big(x^{-1/20}\Big).
$$
Note that 
\begin{equation}\label{BoundPowerLambda}
\sum_{n\leq x}\frac{\Lambda_k(n)}{n}\leq \left(\sum_{n\leq x}\frac{\Lambda(n)}{n}\right)^k\leq (2\log x)^k \ll x^{1/40},
\end{equation}
if $x$ is large enough. Hence, we deduce that
\begin{equation}\label{MomentsMain} \frac{1}{|\F|}\sum_{D\in \Fstar\setminus \E} \left(-\frac{L'}{L}(1,\chi_D)\right)^k
= \frac{1}{|\F|}\sum_{D\in \F} \sum_{n\leq x} \frac{\Lambda_k(n)}{n}\chi_D(n)+O\Big(x^{-1/20}\Big).
\end{equation}
To evaluate the sum on the right hand side of this estimate, we first consider the contribution of perfect squares, which gives the main term. In this case, we use the following standard estimate (see for example \cite{GrSo})
$$ \sum_{D\in \F} \chi_D(m^2)=\sum_{\substack{D\in \F\\ (D,m)=1}}1=\frac{6}{\pi^2}x\prod_{p|m}\left(\frac{p}{p+1}\right)+ O\big(x^{1/2} d(m)\big),$$
where $d(m)$ is the divisor function.
Therefore the contribution of the terms $n=m^2$ to the right hand side of \eqref{MomentsMain} equals
\begin{equation}\label{diagonalChar}
\sum_{m\leq \sqrt{x}}
\frac{\Lambda_{k}(m^2)}{m^2}\prod_{p|m}\left(\frac{p}{p+1}\right)+O\left(x^{-1/2}\sum_{m\leq \sqrt{x}}
\frac{\Lambda_{k}(m^2)}{m^2}d(m)\right).
\end{equation}
By \eqref{boundLambda}, the error term in the last estimate is
\begin{equation}\label{Error2.6}
\ll x^{-1/2}\sum_{m\leq \sqrt{x}}\frac{(2\log m)^k d(m)}{m^2}\leq x^{-1/2} (\log x)^k\sum_{m=1}^{\infty}\frac{d(m)}{m^2}\ll x^{-1/4}.
\end{equation}
Further, since the function $(\log t)^k/\sqrt{t}$ is decreasing for $t\geq e^{2k}$,   we obtain
$$ \sum_{m> \sqrt{x}}
\frac{\Lambda_{k}(m^2)}{m^2}\prod_{p|m}\left(\frac{p}{p+1}\right)\leq \sum_{m> \sqrt{x}}
\frac{(2\log m)^k}{m^2} \ll \frac{(\log x)^{k}}{\sqrt{x}}\sum_{m>\sqrt{x}}\frac{1}{m^{3/2}}\ll \frac{(\log x)^{k}}{x}\ll x^{-1/2}.$$
Thus, combining this bound with \eqref{diagonalChar} and \eqref{Error2.6} we deduce that the contribution of the squares to the right hand side of \eqref{MomentsMain} is 
\begin{equation}\label{squares}
 \sum_{m=1}^{\infty}
\frac{\Lambda_{k}(m^2)}{m^2}\prod_{p|m}\left(\frac{p}{p+1}\right) +O\left(x^{-1/4}\right).
\end{equation}

To bound the contribution of the non-squares, we use the following simple application of the P\'olya-Vinogradov inequality, which corresponds to Lemma 4.1 of \cite{GrSo} and states that
$$ \sum_{D\in \F} \chi_D(n)\ll x^{1/2}n^{1/4}\log n,$$
if $n$ is not a perfect square. 
Using this bound along with \eqref{BoundPowerLambda}, we deduce that the contribution of the non-squares  to the right hand side of \eqref{MomentsMain} is 
\begin{equation}\label{nonsquares}
 \ll x^{-1/4}\log x\sum_{n\leq x}\frac{\Lambda_{k}(n)}{n}\ll x^{-1/4}(2\log x)^{k+1}\ll x^{-1/6}.
\end{equation}
Furthermore, it follows from Lemma \ref{BoundLD} along with the classical zero free region for $L(s, \chi_D)$ that for $D\in \Fstar$ we have 
\begin{equation}\label{ClassicalBound}
\frac{L'}{L}(1,\chi_D)\ll (\log |D|)^2.
\end{equation}
Therefore, combining this bound with equations \eqref{MomentsMain}, \eqref{squares} and \eqref{nonsquares} we derive
\begin{align*}
\frac{1}{|\F|}\sum_{D\in \Fstar} \left(-\frac{L'}{L}(1,\chi_D)\right)^k
&= \frac{1}{|\F|}\sum_{D\in \Fstar\setminus \E} \left(-\frac{L'}{L}(1,\chi_D)\right)^k +O\left(x^{-1/4}(\log x)^{2k}\right)\\
&=\sum_{m=1}^{\infty}
\frac{\Lambda_{k}(m^2)}{m^2}\prod_{p|m}\left(\frac{p}{p+1}\right)+ O \big(x^{-1/20}\big).
\end{align*}
\end{proof}
We are now ready to prove Theorem \ref{MomentsGamma}.
\begin{proof}[Proof of Theorem \ref{MomentsGamma}]
Note that for any prime $p$ and positive integer $k$ we have 
$$ \ex\left((X(p)^k\right)= \frac{p}{2(p+1)}+ (-1)^k\frac{p}{2(p+1)}.$$ Therefore, by the independence of the random variables $X(p)$ we deduce 
that 
\begin{equation}\label{ortho}
\ex\big(X(n)\big)=\begin{cases} \prod_{p|n} \left(\frac{p}{p+1}\right) & \text{if } n \text{ is a square},\\
0 & \text{otherwise}.\\
\end{cases}
\end{equation}
Hence, we obtain
$$\mathbb{E} \left(\left( \sum_{n=1}^{\infty} \frac{\Lambda(n)X(n)}{n} \right)^{k}\right)=\ex\left( \sum_{n=1}^{\infty} \frac{\Lambda_k(n)X(n)}{n}\right)=\sum_{m=1}^{\infty}
\frac{\Lambda_{k}(m^2)}{m^2}\prod_{p|m}\left(\frac{p}{p+1}\right).$$
Therefore, it follows from Theorem \ref{MomentsEuler} that
\begin{align*}
\frac{1}{|\F|} \sum_{D\in \Fstar} (\g)^k&= \frac{1}{|\F|} \sum_{D\in \Fstar}\left(\gamma+
\frac{L'}{L}(1,\chi_D)\right)^k\\
&= \sum_{j=0}^k \binom{k}{j} \gamma^{k-j}\frac{1}{|\F|} \sum_{D\in \Fstar} \left(\frac{L'}{L}(1,\chi_D)\right)^j\\
&= \sum_{j=0}^k \binom{k}{j} \gamma^{k-j} (-1)^j \ex\left(\left( \sum_{n=1}^{\infty} \frac{\Lambda(n)X(n)}{n} \right)^j\right)+O\left(x^{-1/30}\right)\\
&= \ex\left((\gr)^k\right)+O\left(x^{-1/30}\right).
\end{align*} 
\end{proof}

\section{The Laplace transform of $\g$: proof of Theorem \ref{AsympLaplace}}

In order to obtain Theorem \ref{AsympLaplace} from Theorem \ref{MomentsGamma}, we need a uniform bound for the moments of $\gr$. We prove

\begin{pro}\label{BoundMomRand} There exists a constant $c>0$ such that for all positive integers $k\geq 8$ we have 
$$ \ex\left(\left|\gr\right|^k\right)\leq \big(c\log k\big)^k.$$

\end{pro}

\begin{proof} Let $y>2$ be a real number to be chosen later. By Minkowski's inequality we have 
\begin{equation}\label{Minkowski}
\begin{aligned}
\ex\left(\left|\gr\right|^k\right)^{1/k}
&\leq \ex\left(\left|\gamma-\sum_{n\leq y}\frac{\Lambda(n)X(n)}{n}\right|^k\right)^{1/k} +\ex\left(\left|\sum_{n>y}\frac{\Lambda(n)X(n)}{n}\right|^k\right)^{1/k}\\
&\leq \gamma+\sum_{n\leq y}\frac{\Lambda(n)}{n}+\ex\left(\left|\sum_{n>y}\frac{\Lambda(n)X(n)}{n}\right|^k\right)^{1/k}.\\
&\ll \log y+\ex\left(\left|\sum_{n>y}\frac{\Lambda(n)X(n)}{n}\right|^k\right)^{1/k}.\\
\end{aligned}
\end{equation}
Furthermore, by the Cauchy-Schwarz inequality we have 
\begin{equation}\label{Cauchy}
\ex\left(\left|\sum_{n>y}\frac{\Lambda(n)X(n)}{n}\right|^k\right) \leq \ex\left(\left(\sum_{n>y}\frac{\Lambda(n)X(n)}{n}\right)^2\right)^{1/2} \ex\left(\left(\sum_{n>y}\frac{\Lambda(n)X(n)}{n}\right)^{2(k-1)}\right)^{1/2}
\end{equation}
Let $$\Lambda_{\ell,y}(n):=\sum_{\substack{n_1,n_2,\dots, n_{\ell}>y\\ n_1n_2\cdots n_{\ell}=n}} \Lambda(n_1)\Lambda(n_2)\cdots \Lambda(n_{\ell}).$$
Then, for every positive integer $m$ we have 
\begin{align*}
\ex\left(\left(\sum_{n>y}\frac{\Lambda(n)X(n)}{n}\right)^{2m}\right)
&= \ex\left(\sum_{n>y^{2m}}\frac{\Lambda_{2m,y}(n)X(n)}{n}\right)\\
&=\sum_{n>y^{m}}\frac{\Lambda_{2m,y}(n^2)}{n^2}\prod_{p|n}\left(\frac{p}{p+1}\right)\\
&\leq 
\sum_{n>y^{m}}\frac{(2\log n)^{2m}}{n^2},
\end{align*}
since $\Lambda_{\ell,y}(n)\leq \Lambda_{\ell}(n)\leq (\log n)^{\ell}$. Moreover, since $(\log n)^{2m}/\sqrt{n}$ is decreasing for $n>e^{4m}$, we deduce that if $y\geq e^4 $ then 
$$ \ex\left(\left(\sum_{n>y}\frac{\Lambda(n)X(n)}{n}\right)^{2m}\right)\leq \frac{(2m\log y)^{2m}}{y^{m/2}}\sum_{n>y^{m/2}}\frac{1}{n^{3/2}}\ll\frac{(2m\log y)^{2m}}{y^{m}}.$$
Thus if $y\geq e^4 $ then by \eqref{Cauchy} we obtain that 
$$ \ex\left(\left|\sum_{n>y}\frac{\Lambda(n)X(n)}{n}\right|^k\right)^{1/k}\ll \frac{k\log y}{\sqrt{y}}.$$
Choosing $y=k^2$ and inserting this estimate in \eqref{Minkowski} completes the proof.
\end{proof}

\begin{proof}[Proof of Theorem \ref{AsympLaplace}] Given $\epsilon>0$, it follows from Corollary \ref{ASBound} that there exists a constant $B_{\epsilon}>0$ such that 
$$ |\g|\leq B_{\epsilon} \log\log x,$$
for all fundamental discriminants $D\in \F$ except for a set $\E$ with $|\E|=O\left(x^{\epsilon}\right).$
Let $N=\lfloor \log x/(50\log\log x)\rfloor$. Then we obtain
\begin{equation}\label{TaylorLaplace}
\frac{1}{|\F|}\sum_{D\in \F\setminus \mathcal{E}(x)}\exp\left(s \cdot \g\right)
= \sum_{k=0}^N \frac{s^k}{k!} \frac{1}{|\F|}\sum_{D\in \F\setminus \mathcal{E}(x)} \left(\g\right)^k+E_2\\
\end{equation}
where
$$ E_2\ll \sum_{k>N} \frac{|s|^k}{k!} (B_{\epsilon}\log\log x)^k\leq \sum_{k>N} \left(\frac{3B_{\epsilon}|s|\log\log x}{N}\right)^k \ll e^{-N} $$
by Stirling's formula, if $|s|\leq C_{\epsilon}\log x/(\log\log x)^2$ for some small constant $C_{\epsilon}>0$. Furthermore, it follows by Theorem \ref{MomentsGamma} and equation \eqref{ClassicalBound} that for all integers $0\leq k\leq N$ we have 
\begin{align*}
\frac{1}{|\F|}\sum_{D\in \F\setminus \mathcal{E}(x)} \left(\g\right)^k&=\frac{1}{|\F|}\sum_{D\in \Fstar} \left(\g\right)^k +O\left(x^{-1+\epsilon}(\log x)^{2k}\right)\\
&=\ex\left(\gr^k\right) +O\left(x^{-1/20}\right).
\end{align*}
Moreover, it follows from Proposition \ref{BoundMomRand} and Stirling's formula that for some positive constant $C$ we have 
$$ \sum_{k>N} \frac{s^k}{k!} \ex\left(\gr^k\right) \ll \sum_{k>N} \left(\frac{C|s|\log k}{k}\right)^k\ll \sum_{k>N} \left(\frac{C|s|\log N}{N}\right)^k\ll e^{-N},$$
if $C_{\epsilon}$ is suitably small.
Finally, inserting these estimates in \eqref{TaylorLaplace},  we derive
\begin{align*}
\frac{1}{|\F|}\sum_{D\in \F\setminus \mathcal{E}(x)}\exp\left(s \cdot \g\right)&=  \sum_{k=0}^N \frac{s^k}{k!} \ex\left(\gr^k\right) +O\left(e^{-N}+ x^{-1/20}e^{|s|}\right)\\
&= \ex\Big(\exp\big( s \cdot \gr\big)\Big)+  O\left(e^{-N}\right),
\end{align*}
as desired.
\end{proof}


\section{The Laplace transform of $\gr$}

\noindent For any $s\in \mathbb{C}$ we define
$$  M(s):= \log\left(\ex\Big(\exp\big(s\cdot\gr\big) \Big)\right).$$
Since the $X(p)$ are independent and 
$ \gr= \gamma-\sum_{p}(\log p)X(p)/(p-X(p))$ we deduce that 
\begin{equation}\label{LaplaceProduct}
M(s)= \gamma s+ \sum_{p}\log h_p(s),
\end{equation}
where
$$h_p(s):= \ex\left(\exp\left(-\frac{s(\log p) X(p)}{p-X(p)}\right)\right).$$
Note that 
\begin{equation}\label{expectation}
h_p(s)= \frac{p}{2(p+1)}\exp\left(\frac{s\log p}{p+1}\right)+ \frac{p}{2(p+1)}\exp\left(\frac{-s\log p}{p-1}\right)+  \frac{1}{p+1}.
\end{equation}

The main purpose of this section is to investigate the asymptotic behavior of $M(r)$ and its derivatives,  where $r$ is a large real number. We establish the following proposition.

\begin{pro}\label{LaplaceRand}
For any real number $r\geq 4$ we have
\begin{equation}\label{LaplaceRand1}
M(r)=r\left(\log r+\log\log r+A_1-1+O\left(\frac{\log\log r}{\log r}\right)\right),
\end{equation}
\begin{equation}\label{LaplaceRand2}
M(-r)=r\left(\log r+\log\log r+A_2-1+O\left(\frac{\log\log r}{\log r}\right)\right),
\end{equation}
\begin{equation}\label{LaplaceRand3}
M'(r)= \log r+\log\log r +A_1+ O\left(\frac{\log\log r}{\log r}\right),
\end{equation}
and
\begin{equation}\label{LaplaceRand4}
M'(-r)= -\log r-\log\log r -A_2+ O\left(\frac{\log\log r}{\log r}\right).
\end{equation}
Moreover, for all real numbers $y, t$ such that $|y|\geq 3$ we have
\begin{equation}\label{LaplaceRand5}
M''(y)\asymp \frac{1}{|y|}, \text{ and } M'''(y+it)\ll \frac{1}{|y|^2}.
\end{equation}
\end{pro}
To prove this result we first need some preliminary lemmas. 
\begin{lem}\label{estimate1}  Let $r\geq 4$ be a real number. Then we have 

\begin{equation}\label{approxlarge}
\log h_p(r)=\begin{cases} 
r\frac{\log p}{p+1} +O(1) & \text{ if } p\leq r^{2/3}\\
\log\cosh\left(\frac{r\log p}{p+1}\right)+O\left(\frac{r\log p}{p^2}\right) & \text{ if } p>r^{2/3}.
\end{cases}
\end{equation}
and 
\begin{equation}\label{approxlarge2}
\log h_p(-r)=\begin{cases} 
r\frac{\log p}{p-1} +O(1) & \text{ if } p\leq r^{2/3}\\
\log\cosh\left(\frac{r\log p}{p-1}\right)+O\left(\frac{r\log p}{p^2}\right) & \text{ if } p>r^{2/3}.
\end{cases}
\end{equation}

\end{lem}

\begin{proof}
We only prove \eqref{approxlarge} since \eqref{approxlarge2} can be obtained similarly.
First, if $p<r^{2/3}$ then 
\begin{equation}\label{Smallprimes}
h_p(r)= \frac{p}{2(p+1)}\exp\left(\frac{r\log p}{p+1}\right)
\left(1+O\left(\exp\left(-r^{1/3}\right)\right)\right),
\end{equation}
from which the desired estimate follows in this case. 

Now, if $p>r^{2/3}$ then
\begin{equation}\label{locallargep}
\begin{aligned}
h_p(r)
&= \frac{p}{(p+1)} \cosh\left(\frac{r\log p}{p+1}\right) \left(1+O\left(\frac{r\log p}{p^2}\right)\right)+\frac{1}{p+1}\\
&= \cosh\left(\frac{r\log p}{p+1}\right) \left(1+O\left(\frac{r\log p}{p^2}\right)\right),
\end{aligned}
\end{equation}
since $ \cosh(t)-1 \ll  t\cosh(t)$, for all $t\geq 0$. This completes the proof.

\end{proof}

\begin{lem}\label{estimatelogarithmicderivative}  Let $r\geq 4$ be a real number. Then we have 

\begin{equation}\label{approxlarge3}
\frac{h'_p(r)}{h_p(r)}=\begin{cases}  \frac{\log p}{p+1} \left(1+O\left(e^{-r^{1/3}}\right)\right)& \text{ if } p\leq r^{2/3}\\
\frac{\log p}{p+1}\tanh\left(r\frac{\log p}{p+1}\right) + O\left(\frac{\log p}{p^2}+\frac{r\log^2p}{p^3}\right) & \text{ if } p>r^{2/3}.
\end{cases}
\end{equation}
and 
\begin{equation}\label{approxlarge4}
\frac{h'_p(-r)}{h_p(-r)}=\begin{cases}  -\frac{\log p}{p-1} \left(1+O\left(e^{-r^{1/3}}\right)\right)& \text{ if } p\leq r^{2/3}\\
-\frac{\log p}{p-1}\tanh\left(r\frac{\log p}{p-1}\right) + O\left(\frac{\log p}{p^2}+\frac{r\log^2p}{p^3}\right) & \text{ if } p>r^{2/3}.
\end{cases}
\end{equation}

\end{lem}

\begin{proof}
We only prove \eqref{approxlarge3} since the proof of \eqref{approxlarge4} is similar. By \eqref{expectation} we have 
$$ h_p'(r)= 
\frac{p\log p}{2(p+1)^2}\exp\left(\frac{r\log p}{p+1}\right)-\frac{p\log p}{2(p^2-1)}\exp\left(\frac{-r\log p}{p-1}\right)
$$
First, for $p<r^{2/3}$ we have by \eqref{Smallprimes}
$$ h_p'(r)= \frac{\log p}{p+1} h_p(r)\left(1+O\left(\exp\left(-r^{1/3}\right)\right)\right).$$
On the other hand, if $p>r^{2/3}$ then  
$$ h_p'(r)= \frac{\log p}{p+1}\left(\sinh\left(\frac{r\log p}{p+1}\right) + O\left(\frac{1}{p}\cosh\left(\frac{r\log p}{p+1}\right)+\frac{r\log p}{p^2}\right)\right).
$$
Therefore, by \eqref{locallargep} we obtain
$$ 
\frac{h_p'(r)}{h_p(r)}= \frac{\log p}{p+1}\tanh\left(\frac{r\log p}{p+1}\right)+ O\left(\frac{\log p}{p^2}+\frac{r\log^2p}{p^3}\right).
$$

\end{proof}

\begin{lem}\label{sumprimes}
We have 
\begin{equation}\label{sumprimes1}
\sum_{p\leq y} \frac{\log p}{p-1}= \log y-\gamma+O\left(\frac{1}{\log y}\right),
\end{equation}
and 
\begin{equation}\label{sumprimes2}
\sum_{p\leq y} \frac{\log p}{p+1}= \log y - \gamma +2\frac{\zeta'(2)}{\zeta(2)}+ O\left(\frac{1}{\log y}\right).
\end{equation}
\end{lem}

\begin{proof}
We have 
$$\sum_{p\leq y} \frac{\log p}{p-1}= \sum_{p\leq y} \log p\sum_{a=1}^{\infty}\frac{1}{p^a}=\sum_{n\leq y} \frac{\Lambda(n)}{n}+ O\big(y^{-1/2}\big).$$
The first assertion follows from the classical estimate 
\begin{equation}\label{LambdaEstimate}
\sum_{n\leq y} \frac{\Lambda(n)}{n}=\log y-\gamma +O\left(\frac{1}{\log y}\right).
\end{equation}
Moreover, the second assertion follows from the first upon noting that
$$
 \sum_{p\leq y} \frac{\log p}{p+1}
= \sum_{p\leq y} \frac{\log p}{p-1}-2\sum_{p\leq y}\frac{\log p}{p^2-1}= \sum_{p\leq y} \frac{\log p}{p-1} +2\frac{\zeta'(2)}{\zeta(2)}+O\left(\frac{1}{y}\right).
$$

\end{proof}

Let 
$$ f(t):= \begin{cases} \log \cosh(t) & \text{ if } 0\leq t <1 \\
 \log \cosh(t)- t  & \text{ if } t \geq 1.\end{cases}
$$
Then we prove
\begin{lem}\label{logcosh}  $f$ is bounded on $[0,\infty)$ and $f(t)=t^2/2+O(t^4)$ if $0\leq t <1.$ Moreover we have 
\begin{equation}\label{asympldcosh}
 f'(t)=\begin{cases}  t +O(t^2) & \text{ if } 0< t<1 \\
  O(e^{-2t}) & \text{ if } t > 1.\end{cases}
\end{equation}
\end{lem}

\begin{proof}
Since $e^t/2\leq \cosh(t)\leq e^t$, it follows that $f$ is bounded on $[0,\infty)$. Now, for  $t\in [0,1)$ we have  $\cosh(t)=1+t^2/2+O(t^4)$ and hence $f(t)=t^2/2+O(t^4)$. 

Moreover, if $0< t< 1$ then $f'(t)=\tanh(t)=t+O(t^2)$. Now, if $t> 1$ then 
$$f'(t)=\tanh(t)-1= \frac{e^t-e^{-t}}{e^t+e^{-t}}=O(e^{-2t}).$$
\end{proof}

We are now ready to prove Proposition \ref{LaplaceRand}. 
\begin{proof}[Proof of Proposition \ref{LaplaceRand}]
We only prove \eqref{LaplaceRand1} and \eqref{LaplaceRand3}, since \eqref{LaplaceRand2}, \eqref{LaplaceRand4} and \eqref{LaplaceRand5} follow along the same lines.
By Lemma \ref{estimate1} and the prime number theorem we obtain
$$M(r)=\gamma r+\sum_{p\leq r^{2/3}}\frac{r\log p}{p+1}+\sum_{p>r^{2/3}}\log\cosh\left(\frac{r\log p}{p+1}\right)+ O\left(r^{2/3}\right).$$
Let $R$ be the unique solution to  $r\log R=R+1$. Then we have 
$$R= r\log r\left(1+O\left(\frac{\log\log r}{\log r}\right)\right).$$
Since $(\log t)/(t+1)$ is decreasing for $t\geq 4$ we deduce 
$$M(r)=\gamma r+\sum_{p\leq R}\frac{r\log p}{p+1}+\sum_{p>r^{2/3}}f\left(\frac{r\log p}{p+1}\right)+ O\left(r^{2/3}\right).$$
Moreover, by \eqref{sumprimes2}  we have 
$$ \sum_{p\leq R}\frac{\log p}{p+1}= \log R-\gamma+ 2\frac{\zeta'(2)}{\zeta(2)}+ O\left(\frac{1}{\log R}\right)= \log r+\log\log r -\gamma+ 2\frac{\zeta'(2)}{\zeta(2)}+ O\left(\frac{\log\log r}{\log r}\right).$$
Now, by Lemma \ref{logcosh} and the prime number theorem in the form $\pi(t)-\text{Li}(t)\ll t/(\log t)^3$ we derive
\begin{equation}\label{sumintegral}
\begin{aligned}
\sum_{p>r^{2/3}}f\left(\frac{r\log p}{p+1}\right) &=\int_{r^{2/3}}^{\infty} f\left(\frac{r\log t}{t+1}\right)d\pi(t)\\
&= \int_{r^{2/3}}^{\infty} f\left(\frac{r\log t}{t+1}\right)\frac{dt}{\log t} +E_3,
\end{aligned}
\end{equation}
where 
$$E_3\ll r^{2/3} + r\int_{r^{2/3}}^{\infty} \left|f'\left(\frac{r\log t}{t+1}\right)\right|\frac{1}{t(\log t)^2}dt\ll \frac{r}{\log r},$$
since $f'(t)$ is bounded by Lemma \ref{logcosh}. 
To evaluate the main term on the right hand side of \eqref{sumintegral}  we make the change of variables $u= r(\log t)/(t+1).$
Since $t\geq r^{2/3}$ we obtain that
$$
 du= r \left(\frac{1}{t(t+1)}-\frac{\log t}{(t+1)^2}\right)dt
= -r\frac{(\log t)dt}{(t+1)^2}\left(1+O\left(\frac{1}{\log r}\right)\right)
= -\frac{u^2}{r}\frac{dt}{\log t}\left(1+O\left(\frac{1}{\log r}\right)\right).
$$
Putting $r_1= r(\log(r^{2/3}))/(r^{2/3}+1)$, we deduce by Lemma \ref{logcosh} that
\begin{equation}\label{sumintegral2}
 \sum_{p>r^{2/3}}f\left(\frac{r\log p}{p+1}\right) = r \int_0^{r_1} \frac{f(u)}{u^2}du + O\left(\frac{r}{\log r}\right) =  r \int_0^{\infty} \frac{f(u)}{u^2}du + O\left(\frac{r}{\log r}\right).
\end{equation}
Moreover, by a simple integration by parts we have
$$ \int_0^{\infty}\frac{f(u)}{u^2}du=\int_0^{\infty}\frac{f'(u)}{u}du-1.$$
Collecting the above estimates yields \eqref{LaplaceRand1}.

Now, we prove \eqref{LaplaceRand3}. First, note that 
$$ M'(r)= \gamma+\sum_{p} \frac{h_p'(r)}{h_p(r)}.$$
Using Lemma \ref{estimatelogarithmicderivative} we obtain
\begin{align*}
M'(r)&= \gamma+ \sum_{p<r^{2/3}}\frac{\log p}{p+1}+\sum_{p>r^{2/3}}\frac{\log p}{p+1}\tanh\left(\frac{r\log p}{p+1}\right)+ O\left( r^{-1/3}\log r\right)\\
& = \gamma+ \sum_{p<R}\frac{\log p}{p+1}+\sum_{p>r^{2/3}}\frac{\log p}{p+1}f'\left(\frac{r\log p}{p+1}\right)+ O\left( r^{-1/3}\log r\right)\\
&= \log r+\log\log r+ 2\frac{\zeta'(2)}{\zeta(2)}+\sum_{p>r^{2/3}}\frac{\log p}{p+1}f'\left(\frac{r\log p}{p+1}\right)+ O\left( \frac{\log\log r}{\log r}\right).
\end{align*}
Finally, using the prime number theorem and partial integration as in \eqref{sumintegral2}, one can deduce that 
$$
\sum_{p>r^{2/3}}\frac{\log p}{p-1}f'\left(\frac{r\log p}{p-1}\right)=\int_0^{\infty}\frac{f'(u)}{u}du+O\left(\frac{1}{\log r}\right).
$$

\end{proof}


\section{The distribution function of $\gr$: proof of Theorem \ref{ExponentialDecay}}
To shorten our notation, we define $\lapr(s):= \ex\left(\exp\left(s\cdot \gr\right)\right)$. Let $\phi(y)=1$ if $y>1$ and equals $0$ otherwise. 
To relate the distribution function of $\gr$ (or that of $\g$) to its Laplace transform, we use the following smooth analogue of Perron's formula, which is a slight variation of a formula of Granville and Soundararajan (see \cite{GrSo}).

\begin{lem}\label{SmoothPerron}
Let $\lambda>0$ be a real number and $N$ be a positive integer. For any $c>0$ we have for $y>0$
$$
0\leq \frac{1}{2\pi i}\int_{c-i\infty}^{c+i\infty} y^s \left(\frac{e^{\lambda s}-1}{\lambda s}\right)^N \frac{ds}{s} -\phi(y)\leq 
\frac{1}{2\pi i}\int_{c-i\infty}^{c+i\infty} y^s \left(\frac{e^{\lambda s}-1}{\lambda s}\right)^N \frac{1-e^{-\lambda N s}}{s}ds.
$$
\end{lem}

\begin{proof}
For any $y>0$ we have 
$$\frac{1}{2\pi i}\int_{c-i\infty}^{c+i\infty} y^s \left(\frac{e^{\lambda s}-1}{\lambda s}\right)^N\frac{ds}{s}
= \frac{1}{\lambda^N}\int_{0}^{\lambda}\cdots \int_0^{\lambda} \frac{1}{2\pi i} \int_{c-i\infty}^{c+i\infty}
\left(ye^{t_1+ \cdots+ t_N}\right)^s\frac{ds}{s} dt_1\cdots dt_N
$$
so that by Perron's formula we obtain
$$ 
\frac{1}{2\pi i}\int_{c-i\infty}^{c+i\infty} y^s \left(\frac{e^{\lambda s}-1}{\lambda s}\right)^N\frac{ds}{s}
= \begin{cases} = 1 & \text{ if } y\geq 1, \\ \in [0,1] & \text{ if } e^{-\lambda N } \leq y < 1,\\
 =0 & \text{ if } 0<y< e^{-\lambda N }. \end{cases}
$$
Therefore we deduce that 
\begin{equation}\label{indicator}
 \frac{1}{2\pi i}\int_{c-i\infty}^{c+i\infty} y^s e^{-\lambda N s} \left(\frac{e^{\lambda s}-1}{\lambda s}\right)^N \frac{ds}{s} \leq \phi(y)\leq \frac{1}{2\pi i}\int_{c-i\infty}^{c+i\infty} y^s \left(\frac{e^{\lambda s}-1}{\lambda s}\right)^N \frac{ds}{s}
\end{equation}
which implies the result.

\end{proof}
Let $\tau$ be a real number and consider the equation $M'(r)=\tau$ (recall that $M(r)=\log\lapr(r)$). By Proposition \ref{LaplaceRand} it follows that $\lim_{r\to\infty} M'(r)=\infty$ and $\lim_{r\to-\infty} M'(r)=-\infty$. Moreover, a simple calculation shows that $h_p''(r)h_p(r)>(h_p'(r))^2$ for all primes $p$, and hence that $M''(r)>0$. Thus, it follows that the equation $M'(r)=\tau$ has a unique solution $\kappa$.
Using a carefull saddle point analysis we obtain an asymptotic formula 
for $\pr(\gr>\tau)$ in terms of the Laplace transform of $\gr$ evaluated at the saddle point $\kappa$.

\begin{thm}\label{SaddlePoint}
Let $\tau$ be large and $\kappa$ denote the unique solution to $M'(r)=\tau$. Then,  we have
$$\pr(\gr>\tau)= \frac{\lapr(\kappa)e^{-\tau \kappa}}{\kappa \sqrt{2\pi M''(\kappa)}}\left(1+O\left(\frac{1}{ \sqrt{\kappa}} \right)\right).$$
Similarly,  if $\widetilde{\kappa}$ is the unique solution to $M'(-r)=-\tau$ then 
$$ \pr(\gr<-\tau)=  \frac{\lapr(-\widetilde{\kappa})e^{-\tau \widetilde{\kappa}}}{\widetilde{\kappa} \sqrt{2\pi M''(-\widetilde{\kappa})}}\left(1+O\left(\frac{1}{ \sqrt{\widetilde{\kappa}}} \right)\right).$$

\end{thm}

Before proving this theorem, we need to show that $\lapr(r+it)$ is rapidly decreasing in $t$. 

\begin{lem}\label{DecayLaplace} Let $s=r+it \in \mathbb{C}$ where $|r|$ is large. Then, in the range $|t|\geq |r|$ we have 

$$|\lapr(s)|\leq \exp\left(- \frac{|t|}{4\log|t|} \right) \lapr(r).$$

\end{lem}

\begin{proof}
For simplicity we suppose that $r$ and $t$ are both positive. Since $|h_p(s)|\leq h_p(r)$ we obtain
that for any $y\geq 2$ 
\begin{equation}\label{decay1}
\frac{\left|\lapr(s)\right|}{\lapr(r)}\leq \prod_{p>y}\frac{|h_p(s)|}{h_p(r)}.
\end{equation}
Moreover, the same argument leading to \eqref{locallargep} shows that for primes $p>|s|^{2/3}$ we have
$$
h_p(s)= \cosh\left(\frac{s\log p}{p}\right) \left(1+O\left(\frac{|s|\log p}{p^2}\right)\right).
$$
Let $y=t(\log t)^2$. Since $\log\cosh(z)=z^2/2+ O(|z|^4)$ for $|z|\leq 1$, we deduce that for all primes $p>y$ 
$$ 
\frac{h_p(s)}{h_p(r)}
= \exp\left(\frac{(s^2-r^2)(\log p)^2}{2p^2}+ O\left(\frac{t\log p}{p^{2}}+\frac{t^4(\log p)^4}{p^4}\right)\right).
$$ 
Since $\text{Re}(s^2-r^2)=-t^2$, it follows from the prime number theorem and equation \eqref{decay1} that
\begin{align*} 
\frac{\left|\lapr(s)\right|}{\lapr(r)}
&\leq \exp\left(-\frac{t^2}{2}\sum_{p>y}\frac{(\log p)^2}{p^{2}}+ O\left(t\sum_{p>y}\frac{\log p}{p^2}+t^4\sum_{p>y}\frac{(\log p)^4}{p^4}\right)\right)\leq \exp\left(- \frac{t}{4\log t} \right).
\end{align*}

\end{proof} 

\begin{proof}[Proof of Theorem \ref{SaddlePoint}]
We only prove the estimate for $\pr(\gr>\tau)$ since the corresponding asymptotic for $\pr(\gr<-\tau)$ requires only minor modifications.

Let $0<\lambda<1/(2\kappa)$ be a real number to be chosen later. Note that $\gr>\tau$ if and only if $\exp(\gr-\tau)>1$. Therefore, using Lemma \ref{SmoothPerron} with $N=1$ we obtain
\begin{equation}\label{approximation1}
\begin{aligned}
0&\leq \frac{1}{2\pi i}\int_{\kappa-i\infty}^{\kappa+i\infty}\lapr(s)e^{-\tau s}\frac{e^{\lambda s}-1}{\lambda s}\frac{ds}{s}-\pr(\gr>\tau)\\
&\leq \frac{1}{2\pi i}\int_{\kappa-i\infty}^{\kappa+i\infty} \lapr(s)e^{-\tau s} \frac{\left(e^{\lambda s}-1\right)}{\lambda s} \frac{\left(1-e^{-\lambda s}\right)}{s}ds.
\end{aligned}
\end{equation}
Since $\lambda\kappa<1/2$ we have
$|e^{\lambda s}-1|\leq 3 \text{ and } |e^{-\lambda s}-1|\leq 2$. 
Hence, by Lemma \ref{DecayLaplace} we obtain
\begin{equation}\label{error12}
 \int_{\kappa-i\infty}^{\kappa-i\kappa}+ \int_{\kappa+i\kappa}^{\kappa+i\infty}\lapr(s)e^{-\tau s}\frac{e^{\lambda s}-1}{\lambda s}\frac{ds}{s} \ll \frac{e^{-\kappa/(4\log \kappa)}}{\lambda \kappa} \lapr(\kappa)e^{-\tau\kappa},
\end{equation}
and similarly
\begin{equation}\label{error2}
\int_{\kappa-i\infty}^{\kappa-i\kappa}+ \int_{\kappa+i\kappa}^{\kappa+i\infty}\lapr(s)e^{-\tau s} \frac{\left(e^{\lambda s}-1\right)}{\lambda s} \frac{\left(1-e^{-\lambda s}\right)}{s}ds \ll \frac{e^{-\kappa/(4\log \kappa)}}{\lambda \kappa} \lapr(\kappa)e^{-\tau\kappa}.
\end{equation}
Let $s=\kappa+it$. If $|t|\leq \kappa$ then $\left|(1-e^{-\lambda s})(e^{\lambda s}-1)\right|\ll \lambda^2|s|^2$. Since $|\lapr(s)|\leq |\lapr(\kappa)$ we derive
$$
\int_{\kappa-i\kappa}^{\kappa+i\kappa} \lapr(s)e^{-\tau s} \frac{\left(e^{\lambda s}-1\right)}{\lambda s} \frac{\left(1-e^{-\lambda s}\right)}{s}ds \ll \lambda\kappa\lapr(\kappa)e^{-\tau \kappa}.
$$ 
Therefore, combining this estimate with equations \eqref{approximation1}, \eqref{error12} and \eqref{error2} we deduce that
\begin{equation}\label{approximation2}
\pr(\gr>\tau) - \frac{1}{2\pi i}\int_{\kappa-i\kappa}^{\kappa+i\kappa}\lapr(s)e^{-\tau s}\frac{e^{\lambda s}-1}{\lambda s^2} ds \ll \left(\lambda\kappa+\frac{e^{-\kappa/(4\log \kappa)}}{\lambda \kappa}\right)\lapr(\kappa)e^{-\tau\kappa}.
\end{equation}
On the other hand, it follows from equation \eqref{LaplaceRand5} that for $|t|\leq \kappa$ we have 
$$ 
M(\kappa+it)= M(\kappa)+it M'(\kappa)-\frac{t^2}{2}M''(\kappa)+ O\left(|t|^3\frac{1}{\kappa^2}\right).
$$
Also, note that
$$ \frac{e^{\lambda s}-1}{\lambda s^2}=\frac{1}{s}\big(1+O(\lambda \kappa)\big)= \frac{1}{\kappa}\left(1-i\frac{t}{\kappa}+ O\left(\lambda \kappa+\frac{t^2}{\kappa^2}\right)\right).$$
Hence, using that $\lapr(s)=\exp(M(s))$ and $M'(\kappa)=\tau$ we obtain 
\begin{align*}
&\lapr(s)e^{-\tau s}\frac{e^{\lambda s}-1}{\lambda s^2}\\
= &\frac{1}{\kappa}\lapr(\kappa)e^{-\tau\kappa}\exp\left(-\frac{t^2}{2}M''(\kappa)\right) 
\left(1-i\frac{t}{\kappa}+O\left(\lambda\kappa+ \frac{t^2}{\kappa^2}+ |t|^3\frac{1}{\kappa^2}\right)\right).\\
\end{align*}
Thus, we get
\begin{equation}\label{TaylorSaddle}
\begin{aligned}
&\frac{1}{2\pi i}\int_{\kappa-i\kappa}^{\kappa+i\kappa}\lapr(s)e^{-\tau s}\frac{e^{\lambda s}-1}{\lambda s^2} ds\\
=& \frac{1}{\kappa}\lapr(\kappa)e^{-\tau\kappa} \frac{1}{2\pi} \int_{-\kappa}^{\kappa}\exp\left(-\frac{t^2}{2}M''(\kappa)\right)
\left(1+ O\left(\lambda\kappa+ \frac{t^2}{\kappa^2}+ |t|^3\frac{1}{\kappa^2}\right)\right)dt
\end{aligned}
\end{equation}
since the integral involving $it/{\kappa}$ vanishes. Further, since $M''(\kappa)\asymp 1/\kappa$ by \eqref{LaplaceRand5} we derive
$$ 
\frac{1}{2\pi} \int_{-\kappa}^{\kappa}\exp\left(-\frac{t^2}{2}M''(\kappa)\right)dt= \frac{1}{\sqrt{2\pi M''(\kappa)}}\left(1+O\left(e^{-\sqrt{\kappa}}\right)\right),
$$
and 
$$ \int_{-\kappa}^{\kappa}|t|^n\exp\left(-\frac{t^2}{2}M''(\kappa)\right)dt\ll \frac{1}{M''(\kappa)^{(n+1)/2}}\ll \frac{\kappa^{n/2}}{\sqrt{M''(\kappa)}}.$$
Inserting these estimates in \eqref{TaylorSaddle} we deduce that
\begin{equation}\label{main}
\begin{aligned}
&\frac{1}{2\pi i}\int_{\kappa-i\kappa}^{\kappa+i\kappa}\lapr(s)e^{-\tau s}\frac{e^{\lambda s}-1}{\lambda s^2} ds\\
=& \frac{\lapr(\kappa)e^{-\tau\kappa}}{\kappa\sqrt{2\pi M''(\kappa)}}
\left(1+ O\left(\lambda\kappa+ \frac{1}{\sqrt{\kappa}}\right)\right).
\end{aligned}
\end{equation}
Finally,  combining the estimates \eqref{approximation2} and \eqref{main} and choosing $\lambda= \kappa^{-2}$ completes the proof.
\end{proof}

\begin{proof}[Proof of Theorem \ref{ExponentialDecay}]
Again we only prove the estimate for $\pr(\gr>\tau)$, as the corresponding estimate for $\pr(\gr<-\tau)$ can be obtained similarly. By Theorem \ref{SaddlePoint} and equation \eqref{LaplaceRand5}, we have 
$$ \pr(\gr>\tau)= \frac{\lapr(\kappa)e^{-\tau \kappa}}{\kappa \sqrt{2\pi M''(\kappa)}}\left(1+O\left(\frac{1}{ \sqrt{\kappa}} \right)\right)= \exp\Big(M(\kappa)-\tau\kappa+O(\log \kappa)\Big),$$ 
where $\kappa$ is the unique solution to $M'(\kappa)=\tau$. Furthermore, by  \eqref{LaplaceRand3} we have
\begin{equation}\label{EstSaddle1}
\tau= \log\kappa+\log\log\kappa+A_1+O\left(\frac{\log\log \kappa}{\log\kappa}\right),
\end{equation}
and hence we deduce from \eqref{LaplaceRand1} that 
\begin{equation}\label{EstSaddle2}
\pr(\gr>\tau)= \exp\left(-\kappa +O\left(\frac{\kappa \log\log \kappa}{\log \kappa}\right)\right).
\end{equation}
Now, \eqref{EstSaddle1} implies that $\log\kappa=\tau+O(\log \tau)$ and 
$$ \kappa\log \kappa= e^{\tau-A_1}\left(1+O\left(\frac{\log \tau}{\tau}\right)\right).$$
Thus, we obtain 
\begin{equation}\label{OrderSaddle}
 \kappa= \frac{e^{\tau-A_1}}{\tau}\left(1+O\left(\frac{\log \tau}{\tau}\right)\right).
\end{equation}
The result follows upon inserting the estimate \eqref{OrderSaddle} in \eqref{EstSaddle2}. 

\end{proof}


\section{The distribution of extreme values of $\g$: proof of Theorem \ref{MainTheorem} }
By Theorem \ref{AsympLaplace} there exists a constant $B>0$ and 
a set of fundamental discriminants $\E\subset \F$ with $|\E|=O\left(\sqrt{x}\right)$, such that for all complex numbers $s$ with $|s|\leq \log x/(B(\log\log x)^2)$ we have 
\begin{equation}\label{Theorem1.5}
\frac{1}{|\F|}\sum_{D\in \F\setminus \mathcal{E}(x)}\exp\left(s \cdot \g\right) = \lapr(s)+ 
O\left(\exp\left(-\frac{\log x}{50\log\log x}\right)\right).
\end{equation}
To shorten our notation we let
$$\pr_x(\g \in S):=\frac{1}{|\F|}\big|\{D\in \F: \g \in S \}\big|,$$
and 
$$\lap(s)= \frac{1}{|\F|}\sum_{D\in \F\setminus \mathcal{E}(x)}\exp\left(s \cdot \g\right).$$

\begin{proof}[Proof of Theorem \ref{MainTheorem}]
As before,  $\kappa$ denotes the unique solution to $M'(r)=\tau$. Let $N$ be a positive integer and  $0<\lambda<\min\{1/(2\kappa), 1/N\}$ be a real number to be chosen later. 

Let $Y=\log x/(2B(\log\log x)^2)$. If $x$ is large enough then equation \eqref{OrderSaddle} insures that $\kappa\leq Y$. Also, note that \eqref{Theorem1.5} holds for all complex numbers $s=\kappa+it$ with $|t|\leq Y$. 
We consider the integrals
$$
I(\tau)= \frac{1}{2\pi i}\int_{\kappa-i\infty}^{\kappa+i\infty}\lapr(s)e^{-\tau s}\left(\frac{e^{\lambda s}-1}{\lambda s}\right)^N\frac{ds}{s}
$$
and 
$$ 
J_x(\tau)= \frac{1}{2\pi i}\int_{\kappa-i\infty}^{\kappa+i\infty}\lap(s)e^{-\tau s}\left(\frac{e^{\lambda s}-1}{\lambda s}\right)^N\frac{ds}{s}.
$$
 Then, using equation \eqref{indicator} we obtain
\begin{equation}\label{Mellin1}
 \pr(\gr>\tau)\leq I(\tau)\leq \pr(\gr>\tau-\lambda N),
\end{equation}
and 
\begin{equation}\label{Mellin2}
 \pr_x\Big(\g>\tau\Big)+O\left(x^{-1/2}\right)\leq J_x(\tau)\leq \pr_x\Big(\g>\tau-\lambda N\Big)+O\left(x^{-1/2}\right),
\end{equation}
since $|\E|/|\F|\ll x^{-1/2}.$

Further, using that $|e^{\lambda s}-1|\leq 3$ and $|\lapr(s)|\leq \lapr(\kappa)$ we obtain 
\begin{equation}\label{tail1}
\int_{\kappa-i\infty}^{\kappa-iY}+ \int_{\kappa+iY}^{\kappa+i\infty}\lapr(s)e^{-\tau s}\left(\frac{e^{\lambda s}-1}{\lambda s}\right)^N\frac{ds}{s}\ll \frac{1}{N}\left(\frac{3}{\lambda Y}\right)^N\lapr(\kappa)
e^{-\tau \kappa}.
\end{equation}
Similarly, using that $|\lap(s)|\leq \lap(\kappa)$ along with Theorem \ref{AsympLaplace}  we get
\begin{equation}\label{tail2}
\int_{\kappa-i\infty}^{\kappa-iY}+ \int_{\kappa+iY}^{\kappa+i\infty}\lap(s) e^{-\tau s}\left(\frac{e^{\lambda s}-1}{\lambda s}\right)^N\frac{ds}{s}
\ll \frac{1}{N}\left(\frac{3}{\lambda Y}\right)^N\lapr(\kappa)e^{-\tau \kappa}.
\end{equation}
Moreover, note that $|(e^{\lambda s}-1)/\lambda s|\leq 3$, which is easily seen by looking at the cases $|\lambda s|\leq 1$ and $|\lambda s|>1.$ Therefore,  combining  equations \eqref{Theorem1.5}, \eqref{tail1} and \eqref{tail2}  we obtain
\begin{equation}\label{difference1}
J_x(\tau)- I(\tau)\ll \frac1N\left(\frac{3}{\lambda Y}\right)^N\lapr(\kappa)e^{-\tau \kappa}+  \frac{Y}{\kappa}3^N e^{-\tau\kappa} \exp\left(-\frac{\log x}{50\log\log x}\right).
\end{equation}
Furthermore, it follows from Theorem \ref{SaddlePoint} and equation \eqref{LaplaceRand5}   that
\begin{equation}\label{order}
 \pr(\gr>\tau)\asymp\frac{\lapr(\kappa)e^{-\tau \kappa}}{k\sqrt{M''(\kappa)}}
\asymp\frac{\lapr(\kappa)e^{-\tau \kappa}}{\sqrt{\kappa}}.
\end{equation}
Thus, choosing $N=[\log\log x]$ and $\lambda= e^{10}/Y$ we deduce that 
\begin{equation}\label{difference2}
J_x(\tau)- I(\tau)\ll \frac{1}{(\log x)^{5}} \pr(\gr>\tau).
\end{equation}
On the other hand, it follows from Theorem \ref{ExponentialDecay} that 
\begin{equation}\label{shift} 
\begin{aligned}
\pr(\gr>\tau\pm \lambda N)&= \pr(\gr>\tau)\exp\left(O\left(\lambda N \frac{e^{\tau}}{\tau}\right)\right)\\
&= \pr(\gr>\tau)\left(1+O\left(\frac{e^{\tau}(\log\log x)^3}{\tau\log x}\right)\right).\\
\end{aligned}
\end{equation}
Combining this last estimate with \eqref{Mellin1}, \eqref{Mellin2}, and \eqref{difference2} we obtain
\begin{align*}
\pr_x(\g>\tau)&
\leq J_x(\tau)+O\big(x^{-1/2}\big) \\
&\leq I(\tau)+ O\left(\frac{\pr(\gr>\tau)}{(\log x)^{5}}+x^{-1/2}\right)\\
&\leq \pr(\gr >\tau)\left(1+O\left(\frac{e^{\tau}(\log\log x)^3}{\tau\log x}\right)\right)+ O\big(x^{-1/2}\big),
\end{align*}
and 
\begin{align*}
\pr_x(\g>\tau)&
\geq J_x(\tau+\lambda N)+O\big(x^{-1/2}\big) \\
&\geq I(\tau+\lambda N)+ O\left(\frac{\pr(\gr>\tau)}{(\log x)^{5}}+ x^{-1/2}\right)\\
&\geq \pr(\gr>\tau)\left(1+O\left(\frac{e^{\tau}(\log\log x)^3}{\tau\log x}\right)\right)+ O\big(x^{-1/2}\big).\\
\end{align*}
The result follows from these estimates together with the fact that  $\pr(\gr>\tau)\gg x^{-1/4}$ in our range of $\tau$, by Theorem \ref{ExponentialDecay}.

\end{proof}


\end{document}